\documentclass[12pt]{article}

\usepackage{amssymb}
\usepackage{amsfonts}
\usepackage{latexsym}
\usepackage{graphicx}

\usepackage{amsmath}

\usepackage{amsthm}
\usepackage{lscape}

\usepackage[usenames]{color}

\newtheorem{theorem}{Theorem}[section]
\newtheorem{lemma}[theorem]{Lemma}
\newtheorem{definition}[theorem]{Definition}
\newtheorem{corollary}[theorem]{Corollary}
\newtheorem{proposition}[theorem]{Proposition}
\newtheorem{remark}[theorem]{Remark}
\newtheorem{remarks}[theorem]{Remarks}
\newtheorem{example}[theorem]{Example}

\newcommand{\Matrix}[1]{\ensuremath{\left[\begin{array}{ccccccccccccccccccccccccr} #1 \end{array}\right]}}
\newcommand{\ii}{\mathrm{i}}
\newcommand{\Z}{\mathbb{Z}}
\newcommand{\R}{\mathbb{R}}
\newcommand{\sq}{\mbox{\footnotesize $\frac{1}{4}$}}
\newcommand{\D}{\mathbb{D}}
\newcommand{\shf}{\mbox{\footnotesize $\frac{1}{2}$}}
\newcommand{\tq}{\mbox{\footnotesize $\frac{3}{4}$}}
\newcommand{\Sone}{{\bf S}^1}
\newcommand{\ee}{{\mathrm e}}
\newcommand{\beqn}{\begin{eqnarray*}}
\newcommand{\eeqn}{\end{eqnarray*}}
\newcommand{\GG}{{\mathcal G}}
\newcommand{\Ha}{\mathbb{H}}
\newcommand{\hd}{\mathcal{H}}
\newcommand{\tl}{\mathcal{T}}
\newcommand{\C}{\mathbb{C}}
\newcommand{\DD}{{\mathrm D}}
\newcommand{\ONE}{{\bf 1}}
\newcommand{\Fix}{\mbox{{\rm Fix}}}
\newcommand{\ES}{\mathbb{S}}
\newcommand{\oo}{\mathbb{O}}
\newcommand{\dd}{\mathrm{d}}
\newcommand{\UU}{{\mathcal U}}
\newcommand{\CC}{{\mathcal C}}

\title{Hopf Bifurcation and Phase Patterns\\in Symmetric Ring Networks}
\author{Ian Stewart \\ Mathematics Institute
\\ University of Warwick \\ Coventry CV4 7AL
\\ United Kingdom}

\date{\today}

\begin{document}
\maketitle

\begin{abstract}
Systems of ODEs coupled with the topology of a 
closed ring are common models in biology,
robotics, electrical engineering, and many other areas of science. 
When the component systems and couplings are identical, 
the system has a cyclic symmetry group
for unidirectional rings and a dihedral symmetry group
for bidirectional rings. Hopf bifurcation in equivariant and network dynamics  
predicts the generic occurrence of periodic discrete rotating waves whose phase patterns
are determined by the symmetry group. We review basic aspects of the theory
in some detail and derive general properties of such rings. New results are obtained
characterising the first bifurcation
for long-range couplings and the direction in which
discrete rotating wave states rotate.
\end{abstract}

\section{Introduction}

Oscillatory processes, in which the same behaviour repeats periodically as time
passes, are widespread in the natural world. Examples in the life sciences
include the motion of animals \cite{BG01,GW85,GS73,Sel88}, 
breathing \cite{ACUC16}, the heartbeat \cite{QHGW14}, 
peristalsis in the gut \cite{F08,LTD76}, the sleep-wake cycle \cite{AK19},
networks of neurons \cite{AB89,BH87,D70,FS78,GY79,H86,MSe13,MY72,PS70,RKB20,Sz65},  
gene regulatory networks \cite{BP10,EL00}, and biological development \cite{T52}. 
Such processes have been modelled extensively as networks of oscillators  \cite{A86,AEKV23,EM78,E85,KOW07,KLL12,KM81,PYZ14,PYPT10,PC16,WSMR13,YM08}.
The term `oscillator' tacitly assumes that the component subsystems 
must be capable of sustaining oscillatory dynamics in isolation. We do not 
make this assumption, which is not necessary for most networks, 
so we refer to these subsystems as {\em nodes}.

Similar periodic patterns also arise in the physical sciences, 
for instance robotics \cite{I08,IKLNPB21}, electronic engineering \cite{E11,LC11,DA08}, 
Josephson junction arrays \cite{AGM91}, and celestial mechanics \cite{M81,MS13,S96}.
Often the model equations are based on detailed features of the
underlying physics; for example planetary orbits are modelled using
Newton's laws of gravity and motion. The equations can also be
simplified `toy' models that illustrate mathematical 
features of the physics, such as the `kicked oscillator' model for
Arnold tongues in resonance \cite{GH83}.

A feature of some of these models is the occurrence of {\em phase patterns},
in which distinct nodes have the same time-periodic
waveform subject to a phase shift. Often these phase shifts are simple fractions
of the period, regularly spaced round the ring. We call such a state a
{\em discrete rotating wave}, or a {\em standing wave} when all 
oscillators are synchronised. Other terms in the literature
are {\em rosette} \cite{H86} and {\em ponies on a merry-go-round} \cite{AGM91}.
We discuss the relation between phase patterns and symmetries of the equations
using equivariant and network dynamics and bifurcation theory.

\subsubsection{Hopf Bifurcation}

In this paper we consider networks of
coupled continuous-time dynamical systems (ordinary differential equations, ODEs).
For simplicity we restrict to state variables in real vector spaces, but much of the
theory extends to manifolds \cite{AF10a,AF10b,DL15}. 

Periodic oscillations in a dynamical system can arise through a variety of mechanisms.
A common one is Hopf bifurcation, in which an equilibrium loses
stability as some parameter varies, generating a limit cycle. A necessary
condition is that the linearised equation has nonzero purely imaginary eigenvalues
$\pm \ii \omega$, where $0 \neq \omega \in \R$. Such eigenvalues are said to be {\em critical}.
With additional `nondegeneracy' conditions (simplicity of this eigenvalue,
 no other imaginary eigenvalues, and the `eigenvalue crossing' condition)
 the Hopf Bifurcation Theorem guarantees
the existence of a unique bifurcating branch of periodic solutions \cite{HKW81}.
These conditions are generic in general dynamical systems, but not for symmetric
dynamical systems \cite{GSS88} or those with network structure \cite{GS23}.
Nonetheless, useful analogues of the Hopf Bifurcation Theorem hold in these contexts.
We therefore employ the term `Hopf bifurcation' in a loose sense: the
occurrence of purely imaginary critical eigenvalues.
In this paper we focus on oscillations generated by Hopf bifurcation
in this sense.

\subsubsection{Phase Patterns}

A network system has distinguished observables, namely
the states of its nodes, and these can be compared to each other.
In many cases the oscillations exhibit clear patterns in which
the phases of distinct component units are related in specific ways.
Biological examples include
quadruped gaits \cite{BG01,CS93b,CS94,GW85,GS73}, such as the walk, in which
the limbs typically move in the sequence
\[
\mbox{left rear}\ \to\ \mbox{left front}\ \to\ \mbox{right rear}\ \to\ \mbox{right front}\ \to \cdots
\]
and each limb moves one quarter of the period after the previous one.
Figure~\ref{F:ele_walk_outline} illustrates this pattern in the walk of an elephant.

\begin{figure}[h!]
\centerline{%
\includegraphics[width=0.9\textwidth]{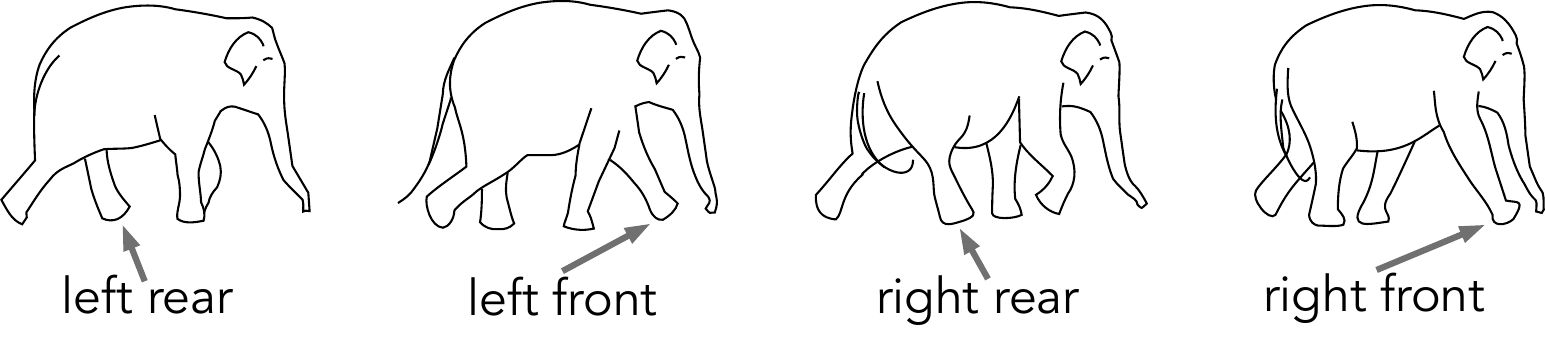}
}
\caption{Walk of an elephant at $\sq$-period intervals.}
\label{F:ele_walk_outline}
\end{figure}

Central configurations in celestial mechanics \cite{M81,MS13,S96} are a physical example.
Suppose that four identical point masses (planets) revolve 
round a central point mass (star), so that the planets follow a circular orbit in the plane 
and lie at the vertices of a uniformly rotating square. Such a state is consistent with
Newtonian gravitation for a suitable choice of the mass and angular velocity.
This state has the same spatiotemporal symmetries as the walk gait: successive planets follow identical orbits with a phase shift of one quarter of the period. Another example in this context is the occurrence of {\em choreographies}, 
such as three identical masses moving along the same figure-8 shaped orbit \cite{CM00}.
An important mathematical difference here is that the equations
of celestial mechanics are Hamiltonian, but analogous results apply \cite{MRS88,MRS90}.
Another is that the centre of mass remains stationary in central configurations;
this happens because the gravitational forces acting on the star 
combine through vector addition, so they cancel out. Such cancellations are common
when interactions are additive.

\subsubsection{Networks and Ring Topology}

A {\em network of coupled ODEs} is a directed graph in which
each node corresponds to a component ODE, and each
arrow (directed edge) indicates that the state of the
node at the head of the arrow is influenced by that of the node at its tail \cite{GS23,GST05,SGP03}. 
In this paper we focus on dynamical systems in which several
subsystems are coupled together with the topology or symmetry of a closed ring. 
Figure \ref{F:rings} illustrates the two main types of ring for identical components with
identical {\em nearest-neighbour} (NN) coupling. We draw each node as a circle and each directed edge as an arrow,
with double-headed arrows for bidirectional coupling. It is convenient to number nodes 
with integers $\{0,1, \ldots, n-1\}$ modulo $n$, where $n$ is the number of nodes
and nodes are numbered consecutively in the anticlockwise direction.

\begin{figure}[h!]
\centerline{%
\includegraphics[width=0.5\textwidth]{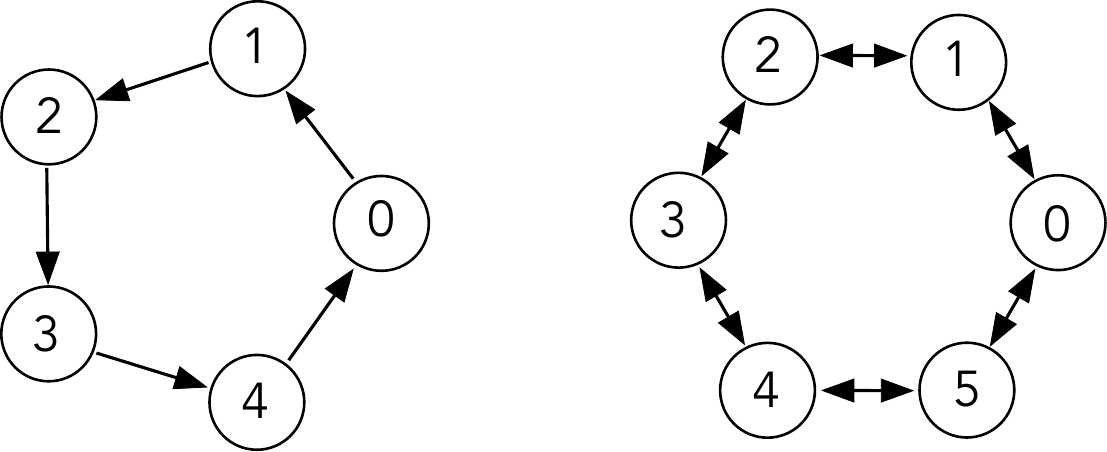}
}
\caption{{\em Left}: 5-node unidirectional ring. {\em Right}: 6-node bidirectional ring.}
\label{F:rings}
\end{figure}

\subsubsection{Symmetry Induces Phase Patterns}
For suitable ODEs, unidirectional and bidirectional rings with identical nodes and arrows
naturally sustain periodic oscillations such that the relative phases of consecutive nodes
differ by the same fraction of the period.
These patterns are a consequence of the {\em symmetries} of the ring: the permutations of nodes that preserve the connecting arrows \cite{GS23,GSS88}.
The symmetry group of an $n$-node unidirectional ring of identical nodes with identical 
NN coupling, as in Figure \ref{F:rings} (left),
is the cyclic group $\Z_n$ of order $n$, which is isomorphic to the group of rotational symmetries of
a regular $n$-gon. 
Considered as a permutation group acting on the nodes,
$\Z_n$ is generated by the $n$-cycle $\alpha = (0\, 1\, 2\, \cdots \, n-1)$.
The symmetry group of an $n$-node bidirectional ring (of identical nodes with identical 
NN coupling as in Figure \ref{F:rings}(right))
is the dihedral group $\D_n$ of order $2n$, which is isomorphic to the group of rotational 
and reflectional symmetries of a regular $n$-gon. The group $\D_n$ is generated by 
$\alpha$ and the map $\kappa$ where $\kappa(j) = -j \pmod{n}$.

More generally, in Section \ref{S:ED} we discuss
oscillations of ODEs with a symmetry group $\Gamma$. The main result
here is the Equivariant Hopf Theorem, which proves the existence of
bifurcating branches of solutions with certain spatiotemporal symmetries,
under suitable conditions.
When applied to ODEs with $n$-fold rotational
symmetry $\Z_n$, this result shows that typical oscillation patterns 
created by Hopf bifurcation have phase shifts
$kT/n$ between successive nodes, where $T$ is the overall period and
$0 \leq k < n$ is an integer: the aforementioned discrete rotating waves. See Theorem \ref{T:EHT}. The predicted patterns for
$\D_n$ symmetry are more complicated; see Section \ref{S:SBR} and \cite{GS86,GS02,GSS88}.

In particular, the $\sq$-period phase pattern in
Figure~\ref{F:ele_walk_outline} can occur in a {\em rigid}
manner in a unidirectional ring of
four identical systems with identical coupling, which has symmetry group $\Z_4$.
By `rigid' we mean that the phase pattern, as a fraction of the period, persists after sufficiently
small perturbations of the ODE that preserve the symmetry and network topology.
This condition is a form of  structural stability \cite{S67}.
Such a ring can also support a phase pattern where the phase shift from one node
to the next is $0, \shf$, or $\tq$ of the period, depending on the form of the ODE
and the solution under consideration. For some ODEs, different patterns may coexist for
different initial conditions. The stability of such states is a complicated issue,
addressed in \cite{GS85a,GS86,GSS88} using Birkhoff normal form reduction.

\begin{remarks}\em

(a) 
Although $\Z_4$ symmetry generates the phase pattern of the walk gait, 
a network that can generate all of the main quadruped gaits, without
distinct gaits being stable simultaneously, has at least 8 nodes and $\Z_2 \times \Z_4$
symmetry \cite{GSBC98,GSCB99}.

(b) Other examples of regular phase patterns include oscillations of the {\em repressilator},
a synthetic genetic circuit \cite{EL00}, and a ring of three identical FitzHugh--Nagumo neurons
\cite{GS23}. Both systems are modelled by ODEs with $\Z_3$ symmetry, and
Hopf bifurcation gives rise to oscillations in which successive components
are related by phase shifts of one third of the period.

(c) In generic $\Z_n$-symmetric Hopf bifurcation, critical eigenvalues are simple and the classical
Hopf Bifurcation Theorem applies. However, the Equivariant Hopf Theorem
provides extra information on the spatiotemporal symmetries of
the oscillations. These symmetries are present in the linearised
eigenfunctions, that is, in the form of the critical eigenvectors, which
proves that the phase pattern exists at linear order. However, the 
Equivariant Hopf Theorem shows that the pattern is {\em exact}, even when
nonlinear (symmetric) terms are present.

(d) In the network context, the situation is not so straightforward.
Equivariant maps for a network with symmetry group $\Gamma$ 
need not be admissible. In consequence the basic theorem that
generically the critical eigenspace for
an imaginary eigenvalue is $\Gamma$-simple is no longer valid.
The reason is that the proof of this theorem involves linear
perturbations given by projection to irreducible components,
which may not be admissible. 
However, for the $\Z_n$-symmetric ring networks studied in this paper,
it can be proved directly that all eigenvalues of generically simple.
\end{remarks}

\subsection{Summary of Paper}

Section \ref{S:BDB} summarises required background material on local
bifurcation, with emphasis on Hopf bifurcation. It discusses the role of symmetry
in dynamics through the notion of an equivariant ordinary differential equation (ODE).

Section \ref{S:END} introduces the main concepts involved in equivariant
dynamics and the analogous theory for network dynamics. In particular we
define the class of admissible ODEs associated with any network; these are
the ODEs that respect the network topology in a specific formal sense.
We review key results from the representation theory of compact
Lie groups, and in particular finite groups, and explain how they apply
to local bifurcation in equivariant dynamical systems. We discuss
symmetry-breaking and state the Equivariant Hopf Theorem on the
existence of bifurcating branches of periodic states for certain
groups of spatio-temporal symmetries. We also mention time reversal symmetry.

Section \ref{S:SUD} applies the Equivariant Hopf Theorem to 
unidirectional ring networks with cyclic symmetry group $\Z_n$. The analysis
leads to conditions for the existence
of discrete rotating waves, in which successive nodes in the ring oscillate
with the same waveform but regularly spaced phases $\frac{kT}{n}$
where $T$ is the overall period and $k$ is an integer between 0 and $n-1$.
Assuming
nodes have 1-dimensional state spaces $\R$, the first bifurcation---the only
one that can be stable locally---is Hopf if and only if the number $n=2N+1$ of nodes is
odd. The associated phase shifts are then $\frac{NT}{2N+1}$ or
$\frac{(N+1)T}{2N+1}$.

Section \ref{S:HBNNUR} specialises the results to networks. We study
Hopf bifurcation in a unidirectional ring with nearest neighbour coupling.
Again we assume nodes have 1-dimensional state spaces $\R$. 
We begin with nearest-neighbour coupling, in which case the results
of Section \ref{S:HBNNUR} impose constraints on the phase shift for
the first Hopf bifurcation. We show that when longer-range couplings are
allowed, the different types of Hopf bifurcation can occur in any order. 

Section \ref{S:MNS} generalises the analysis to nodes with higher-dimensional
state spaces. We show by example that
 the strong constraints on the first bifurcation no longer apply.
 
 Section \ref{S:SBR} carries out similar analyses for bidirectional
 rings, with symmetry group $\D_n$. We also mention `exotic'
 synchrony and phase patterns in networks, which arise for combinatorial
 reasons rather than being associated with symmetries, and discuss
 the difference between equivariant and admissible ODEs.

\section{Background on Dynamics and Bifurcations}
\label{S:BDB}

We assume familiarity with basic concepts in nonlinear dynamics and
bifurcation theory; see \cite{GH83}. In particular we assume knowledge of
the classical Hopf Bifurcation Theorem on the transition from a stable
equilbrium to a branch of periodic states \cite{HKW81}. We summarise analogous
concepts for equivariant dynamics \cite{GSS88} and network dynamics \cite{GS23,GST05,SGP03}. 

\subsection{Local Bifurcation}

Consider a 1-parameter family of smooth ($C^\infty$) maps
$f:\R^n \times \R \rightarrow \R^n$, where $\lambda \in \R$ is a parameter. 
There is a corresponding family of ODEs:
\begin{equation}
\label{E:Bif_ODE}
\dot{x} =  f(x, \lambda)\qquad x \in \R^n, \lambda \in \R
\end{equation}
For simplicity in stating results, we assume that for each $\lambda$, solutions $x_\lambda(t)$
for given initial conditions $x_\lambda(0)$ exist for all $t \in \R$. This condition
is generically valid \cite{LY73}. Usually all we need is local existence for $t$ near $0$,
which holds for all smooth $f$. More technically, everything can be stated for
$t$ lying in a suitable interval.

A {\em branch}  is a parametrised family $\{x_\lambda(t)\}$
of solutions that vary continuously with
$\lambda$. A {\em local bifurcation} occurs at $(x_0,\lambda_0)$ if
the topology of the set of solutions is not constant near $(x_{\lambda_0}(x_0),\lambda_0)$.
There are two types of local bifurcation: steady-state and Hopf.

A necessary condition for the occurrence of local bifurcation from a branch of equilibria
at $(x_0,\lambda_0)$
is that the Jacobian (or derivative) $J = \mathrm{D} f|_{x_0,\lambda_0}$ should have
eigenvalues on the imaginary axis (including 0). These are the {\em critical eigenvalues}
and their real eigenspaces are the {\em critical eigenspaces}. (The real
eigenspace $E_\mu$ for a non-real eigenvalue $\mu$ is the real part of the sum
of the complex eigenspaces for $\mu$ and its conjugate $\bar\mu$. It is spanned
over $\R$ by the real and imaginary parts of the complex eigenvectors.)

\noindent
(1) \quad A zero eigenvalue usually corresponds to {\em steady-state 
bifurcation}: typically, the number of equilibria changes near $(x_0,\lambda_0)$, and branches of
equilibria may appear, disappear, merge, or split as $\lambda$ varies
near $\lambda_0$. The possibilities here
can be organised, recognised, and classified using singularity theory~\cite{GSS88}.

\noindent
(2) \quad A nonzero imaginary eigenvalue usually corresponds to {\em Hopf bifurcation}
\cite{GH83,HKW81}.
Under suitable genericity conditions(a simple pair of eigenvalues $\pm\ii\omega$,
no other imaginary eigenvalues, and the eigenvalue crossing condition)
this leads to periodic
solutions whose amplitude (near the bifurcation point) is small. Degenerate
Hopf bifurcation, where the eigenvalue crossing condition fails, has been analysed
using singularity theory \cite{GL81}.

\subsubsection{Normalised Period and Phase Shifts}
Suppose that $\rho(t)$ is a $T$-periodic function. 
Define the corresponding {\em circle group} to be
$\Sone = \R/T\Z$. This is the group of phase shifts of $\rho$ modulo the period.
It is isomorphic to the multiplicative group of
complex numbers $\ee^{\ii\theta}$ on the unit circle by the map
$\ee^{\ii\theta} \mapsto \frac{T\theta}{2\pi}$, 
and we often use this isomorphism to
identify the two groups.

There are two main conventions when defining a phase shift, differing
in sign. We adopt the following convention:

\begin{definition}\em
\label{D:phaseshift}
Let $\rho(t)$ be a $T$-periodic function and let $\theta \in \Sone= \R/T\Z$. Then 
\begin{equation}
\label{E:phase_shift}
\rho(t - \theta)
\end{equation}
is $\rho(t)$ {\em phase-shifted by} $\theta$. 
Also, $\theta$ is the
{\em phase shift from} $\rho(t)$ {\em to} $\rho(t-\theta)$.
To obtain a unique value it is convenient to normalise $\theta$ to lie in $[0,T)$.
\end{definition}

For fixed $\theta$ and variable $t$, the time series of $\rho(t - \theta)$
is that of $\rho(t)$ shifted $\theta$ to the right.
The alternative convention has a $+$ sign in \eqref{E:phase_shift}.
Now $\theta$ changes to $-\theta$ so the time series is shifted $\theta$ to the left.

For example, in the walk gait the movement of successive legs is
separated by $\sq$-period phase shifts.

\subsection{Symmetries}

A {\em symmetry} of an ODE is a permutation $\sigma$ of the variables
that preserves the ODE. A symmetry of a network is a permutation $\sigma$ of the 
nodes that preserves arrows and their types.

\begin{example}\em
\label{ex:Z31D}
Consider the following ODE in variables $x=(x_0,x_1,x_2)$:
\begin{equation}
\label{E:Z31D}
\begin{array}{rcl}
\dot x_0 &=& \lambda x_0 - x_0^3+a x_2 \\
\dot x_1 &=& \lambda x_1 - x_1^3+a x_0 \\
\dot x_2 &=& \lambda x_2 - x_2^3+a x_1
\end{array}
\end{equation}
where $\lambda$ acts as a bifurcation parameter and $a$ is a coupling strength.

Equation \eqref{E:Z31D} is symmetric under the cyclic permutation $\alpha = (0\,1\,2)$.
That is, if we apply $\alpha$ to the equations we obtain the same set of equations,
though listed in permuted order. Abstractly, the equations are {\em equivariant}
for the group $\Gamma = \Z_3$ generated by $\alpha$. This means that if we write the ODE
as $\dot x = f(x)$, where $x=(x_0,x_1,x_2)$ and $f=(f_0,f_1,f_2)$, then
\begin{equation}
\label{e:Gamma_equiv}
f(\gamma x) = \gamma f(x) \qquad \forall \gamma\in \Gamma
\end{equation}
Here
\[
\alpha x = (x_1,x_2,x_0) \qquad \alpha f = (f_1,f_2,f_0)
\]

Because the functions involved are odd in $x$, 
equation \eqref{E:Z31D} has a further symmetry $\kappa(x)= -x$,
so the full symmetry group is $\Z_3 \times \Z_2$.
We discuss this later in Section \ref{S:ER}. (Adding a small quadratic term
$bx_i^2$ to each equation for $\dot x_i$  removes this extra symmetry,
leading to the same phase pattern.)

It is straightforward to analyse Hopf bifurcation in this example.
There is an equilibrium at $[0,0,0]^\mathrm{T}$ for all $\lambda,a$. The Jacobian is
\[
J = \Matrix{\lambda & 0 & a \\ a & \lambda& 0\\ 0 & a& \lambda}
\]
The eigenvalues $\mu_j$ and corresponding eigenvectors $u_j$ are
\beqn
\mu_0 = \lambda+a &:& [1,1,1]^\mathrm{T} = u_0 \\
\mu_1 = \lambda + \zeta^2a &:& [1,\zeta,\zeta^2]^\mathrm{T}=u_1 \\
\mu_2 = \lambda + \zeta a &:& [1,\zeta^2,\zeta]^\mathrm{T}=u_2
\eeqn
where $\zeta = \ee^{2\pi \ii/3}$. Generically (in $\lambda, a$) these eigenvalues
are simple. Indeed, they are simple unless $a=0$.

Let $\mu_j = \rho_j+\ii \sigma_j$. 
The real and imaginary parts of the eigenvalues are
\beqn
\rho_0 &=& \lambda+a \qquad\quad \sigma_0 = 0 \\
\rho_1 &=& \lambda-a/2 \qquad \sigma_1 = -\frac{\sqrt{3}a}{2}\\
\rho_2 &=& \lambda-a/2 \qquad \sigma_2 = \frac{\sqrt{3}a}{2}
\eeqn
Thus there is a Hopf bifurcation when $a=2\lambda$ and $a \neq 0$.

The trivial solution is stable for $\lambda+a < 0$ and $2\lambda < a$.
The first instability is the Hopf bifurcation if, in addition, $\lambda-a/2 > \lambda+a$, that is, $a<0$.
\end{example}

\begin{figure}[h!]
\centerline{%
\includegraphics[width=0.6\textwidth]{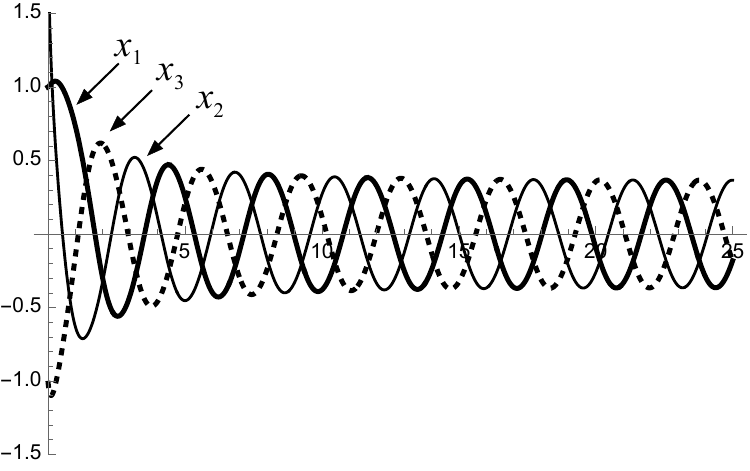}
}
\caption{Periodic state created by Hopf bifurcation in a $\Z_3$-symmetric ring. 
Solid thick = $x_1$, solid thin = $x_2$, dashed = $x_3$.
Parameters $a_0= -0.9, a_1 = -2$ (just after bifurcation).}
\label{F:Z3hopf}
\end{figure}

Figure \ref{F:Z3hopf} shows a simulation for parameter values
$a_0= -0.9, a_1 = -2$.
The oscillations have a clear phase pattern of the form 
\[
x(t) = (u(t), u(t-2T/3),u(t-T/3))
\]
where $T$ is the period and $u(t)$ is the common waveform at each node. 
The time-reversal of this ODE 
yields a rotating wave in the opposite direction:
\[
x(t) = (u(t), u(t-T/3),u(t-2T/3))
\]
In Section \ref{S:SUD} we discuss how the
rotating wave structure is generated by the $\Z_3$ symmetry,
and is related to the form of the eigenvectors $u_1$ and $u_2$.

\section{Equivariant and Network Dynamics}
\label{S:END}

The dynamics and bifurcations of
ring networks can be studied within two different frameworks: 
equivariant dynamics \cite{GSS88} and network dynamics \cite{GS23}.
The first focuses on symmetry properties; the second on the topology
(and type) of connections. Neither alone captures the behaviour of ring networks.
In particular, generic dynamics for a network with symmetry group $\Gamma$
can differ from generic $\Gamma$-equivariant dynamics \cite{GS23}. Symmetry and
topology can interact to create new generic phenomena.
In practice, symmetric networks are studied using a combination
of equivariant and network dynamics, while bearing in mind the potential
for unusual behaviour.

However, equivariant dynamics suffices for the purposes of this paper,
with one additional constraint: the variables that appear in model ODEs
must respect the network topology; see Section \ref{S:AOND}. 
We therefore refer to \cite{GST05,GS23,SGP03} for information on 
general network dynamics, and use the equivariant approach.
Symmetric ring networks are quite well behaved in this regard,
as we summarise briefly in Section \ref{S:SBR}.

Hopf bifurcation for systems with dihedral group symmetry $\D_n$
is analysed in \cite{GS86}; see also \cite{GS02, GSS88}. 
Hopf bifurcation for systems with cyclic group symmetry $\Z_n$
iis mentioned briefly in \cite[Chapter XVII Section 8(b)]{GSS88}. 
We supply further detail and prove some new results.

\subsection{Equivariant Dynamics}
\label{S:ED}

Equivariant dynamics examines how the symmetries of a differential
equation affect the behavior of its solutions, especially
their symmetries. To describe the main results we require some basic concepts. 

For simplicity, assume that the state space of the system is $X = \R^n$
and consider an ODE
\begin{equation}
\label{E:ODEapp}
\dot{x} = f(x) \qquad x \in X
\end{equation}
where $f:X \rightarrow X$ is a smooth map (vector field).
Symmetries enter the picture when a group of linear transformations 
$\Gamma$ acts on $X$. We require all elements of $\Gamma$
to map solutions of the ODE to solutions. By~\cite[Section 1.2]{GS02} this is equivalent to
$f$ being $\Gamma$-{\em equivariant}; that is:
\begin{equation}
\label{E:equivar}
f(\gamma x) = \gamma f(x)
\end{equation}
for all $\gamma \in \Gamma, x \in X$. We call~\eqref{E:ODEapp} a
$\Gamma$-{\em equivariant ODE}.

Condition~\Ref{E:equivar} captures the structure
of ODEs that arise when modeling a symmetric real-world system. 
It states that the vector field inherits the
symmetries, in the sense that symmetrically related points in state
space have symmetrically related vectors. 

For bifurcation theory, we consider a {\em parametrised family} of ODEs
\begin{equation}
\label{E:ODEappB}
\dot{x} = f(x,\lambda) \qquad x \in X
\end{equation}
The equivariance condition is then:
\begin{equation}
\label{E:equivar_lambda}
f(\gamma x,\lambda) = \gamma f(x,\lambda)
\end{equation}
for all $\gamma \in \Gamma, x \in X$. We call~\eqref{E:ODEappB} a
$\Gamma$-{\em equivariant family of ODEs}.

\subsection{Admissible ODEs in Network Dynamics}
\label{S:AOND}

A network of coupled ODEs is represented by a directed graph (or digraph) in which:
(a)
Each {\em node} corresponds to a component ODE. Nodes are classified into
distinct {\em node-types}.

(b)
Each {\em arrow} (directed edge) indicates {\em coupling}: the node at the head of the arrow
receives an input from the node at the tail end. Arrows are classified into
distinct {\em arrow-types}.

We call such a system of
ODEs a {\em network system}. The directed graph is the
 {\em network diagram}. The head of node $c$
is denoted by $\hd(c)$ and the tail by $\tl(c)$.

For each node $c$ choose a {\em node space} $P_c = \R^{n_c}$ and
corresponding  {\em node coordinates} $x_c$ on $P_c$. These coordinates are
multidimensional if $\dim P_c > 1$.
Nodes of the same {\em state-type}  have the same coordinate system
and their dynamics can meaningfully be compared.
The {\em total state space} is $P = \oplus_c P_c$.

A map $f = (f_1, \ldots, f_n)$ from $P$ to itself can be written in components as
\[
f_c: P \rightarrow P_c \qquad 1 \leq c \leq n
\]
For admissibility we impose extra conditions on the $f_c$ that reflect network
architecture. See \cite{GST05,GS23,SGP03} for details; here we summarise the main idea.

\begin{definition}\em
\label{D:admiss}
A map $f: P \rightarrow P$ is $\GG$-{\em admissible} if:

(1) {\em Domain Condition}:  For  every node $c$, the component $f_c$ 
depends only on the node variable $x_c$ and the input variables $x_i$ where
$i$ runs through the tail nodes of arrows with head $c$; see equation \eqref{E:tailnodes} below.

(2) {\em Symmetry Condition}:  If $c$ is a node, $f_c$ is
invariant under all permutations of tail node coordinates for input arrows of the same type.

(3) {\em Pullback Condition}:  If nodes $c$ and $d$ have the same node-type,
and the same number of input arrows of any given type, then
the components $f_c, f_d$ are identical as functions. The variables to
which they are applied correspond under some (hence any, by condition (2)) 
bijection that preserves the arrow-types.
\end{definition}

Conditions (2) and (3) can be combined into
a single {\em pullback condition} applying to any pair $c, d$ of nodes,
but it is convenient to separate them.

Each admissible map $f$ determines an {\em admissible ODE}
\begin{equation}
\label{E:admissODE}
\dot x = f(x)
\end{equation}
In node coordinates this takes the form
\begin{equation}
\label{E:tailnodes}
\dot x_c = f_c(x_c, x_{i_1}, \ldots, x_{i_m})
\end{equation}
for all nodes $c$, where $i_1, \ldots, i_m$ are the tails of the arrows
with head $c$; that is, the {\em input variables} to node $c$.

Equation \eqref{E:Z31D} is an example of an admissible ODE
for a 3-node unidirectional NN coupled ring in which all nodes have the same node-type
and all arrows have the same arrow-type (which is equivalent to
$\Z_3$ symmetry in this case). In particular, the only variables
appearing in the component for $\dot x_c$ are $x_c$ itself and
$x_{c-1}$, the variable corresponding to the tail of the unique input arrow.

If $f$ also depends on a (possibly multidimensional)
parameter $\lambda$, and is admissible as a function
of $x$ for any fixed $\lambda$, we have an {\em admissible family} of maps $f(x,\lambda)$
and ODEs $\dot x =f(x,\lambda)$. Such families arise in bifurcation theory.

If the network has symmetry group $\Gamma$, every admissible map
is $\Gamma$-equivariant. However, the converse is false in general; see \cite[Section 16.5]{GS23} and Section \ref{S:SBR}.

\subsection{Local Bifurcation in Equivariant Dynamics}

\subsubsection{Isotropy Subgroups and Fixed-Point Subspaces}
\label{S:ISFPS}
Recall that if $x(t)$ is a solution of a $\Gamma$-equivariant ODE, then so is $\gamma x(t)$
for all $\gamma \in \Gamma$.
Suppose that an equilibrium $x_0$ is {\em unique}. Then $\gamma x_0$
is also an equilibrium, so by uniquenesss,
$\gamma x_0 = x_0$ for all $\gamma \in \Gamma$. Thus
the solution is symmetric under $\Gamma$. 
However, when uniqueness fails---which is common in nonlinear 
dynamics---equilibria of a $\Gamma$-equivariant ODE
need not be symmetric under the whole of $\Gamma$. This phenomenon, called
{\em (spontaneous) symmetry-breaking}, is a general mechanism
for pattern formation. 

\subsubsection{Symmetry of a Solution}

To formalise `symmetry of a solution' we introduce a key concept:

\begin{definition}\em
If $x \in X$, the {\em isotropy subgroup} of $x$ is
\[
\Sigma_x = \{ \sigma \in \Gamma : \sigma x = x \}
\]
This group consists of all $\sigma$ that fix $x$.

Similarly 
the {\em isotropy subgroup} of a solution $x(t)$ is
\[
\Sigma_{x (t)} = \{ \sigma \in \Gamma : \sigma x(t) = x(t)\ \forall t  \}
\]

\begin{example}\em
\label{ex:Z3revisit}
We revisit Example \ref{ex:Z31D} with $\Gamma = \Z_3 \times \Z_2$.
Solutions on the {\em trivial branch} $x=0$ have isotropy subgroup $\Gamma$. 
Solutions on the bifurcating branch of periodic oscillations
have isotropy subgroup $\Z_2$.
\end{example}

There is a `dual' notion. If $\Sigma \subseteq \Gamma$ is a subgroup of $\Gamma$,
its {\em fixed-point subspace} is
\[
\Fix(\Sigma) = \{ x \in X : \sigma x = x \quad \forall \sigma \in \Sigma \}
\]
\end{definition}
Clearly $\Fix(\Sigma)$ comprises all points $x \in X$ whose isotropy
subgroup contains $\Sigma$. Fixed-point spaces provide
a natural class of subspaces that
are invariant for any $\Gamma$-equivariant map $f$:

\begin{proposition}
\label{P:fixinv}
Let $f:X \rightarrow X$ be a $\Gamma$-equivariant map, and
let $\Sigma$ be any subgroup of $\Gamma$. Then $\Fix(\Sigma)$
is an invariant subspace for $f$, and hence for the dynamics of~\eqref{E:ODEapp}.
\end{proposition}
The proof is so simple that we give it here:
\begin{proof}
If $x \in \Fix(\Sigma)$ and $\sigma \in \Sigma$, then 
\[
\sigma f(x) = f(\sigma x) = f(x)
\]
so $f(x) \in \Fix(\Sigma)$.
\end{proof}

\begin{corollary}
\label{C:}
The isotropy subgroup of a solution 
$x(t)$ is the same as that of any point $x(t_0)$ on it.
\end{corollary}
\begin{proof}
If $x(t_0) \in Y$ where $Y$ is an invariant subspace,
then $x(t) \in Y$ for all $t$.
\end{proof}

We can interpret $\Fix(\Sigma)$ as the space of all states
that have symmetry (at least) $\Sigma$. Then the restriction
\[
f|_{\Fix(\Sigma)}
\]
determines the dynamics of all such states. In particular, we can
find states with a given isotropy subgroup $\Sigma$ by 
considering the (generally) lower-dimensional system determined
by $f|_{\Fix(\Sigma)}$.

If $x \in X$ and $\gamma \in \Gamma$,
the isotropy subgroup of $\gamma x$ is conjugate to that of $x$:
\[
\Sigma_{\gamma x} = \gamma \Sigma_x \gamma^{-1}
\]
Therefore isotropy subgroups occur in conjugacy classes, which correspond
to group orbits of solutions.
For many purposes we can consider isotropy subgroups only up to conjugacy.
The conjugacy classes of isotropy subgroups are ordered by
inclusion (up to conjugacy). The resulting partially ordered set
is called the {\em lattice of isotropy subgroups}
or {\em isotropy lattice} \cite{GSS88}, although technically
it need not be a lattice. (If we do not pass to conjugacy classes, it is a lattice,
and it is often a lattice when we do.)

\subsection{Review of Representation Theory}

The structure of equivariant ODEs is tightly constrained by the representation
of the symmetry group on the state space.
In this paper all symmetry groups 
are compact Lie groups (especially finite groups)
$\Gamma$ acting linearly on $X = \R^k$ for some finite $k$. That is,
if $x \in X, \gamma \in \Gamma$ there are linear maps 
$\rho_\gamma:X \to X$ such that
\[
\rho_1 = I_k \qquad \rho_{\gamma\delta}(x) = \rho_{\gamma}(\rho_\delta(x))
\]
for all $x \in X, \gamma,\delta \in \Gamma$, where $I_k$ is the identity matrix
and $1$ is the identity element of $\Gamma$.
We write
\[
\gamma x = \rho_\gamma(x)
\]
and use the term
`representation' both for the action $\rho$ and for the representation space $X$.
If $Y \subseteq X$ is $\Gamma$-invariant, so that $\gamma y \in Y$ whenever $y \in Y$,
the restriction $\rho|_Y$ of the action of $\Gamma$ to $Y$ is also a representation,
called a {\em subrepresentation} or {\em component}
of $X$.
A representation $X$ is {\em irreducible} if the only $\Gamma$-invariant subspaces
are $0$ and $X$. In the compact case, any
representation $X$ is {\em completely reducible}:
\[
X = X_1 \oplus \cdots \oplus X_m
\]
where the $X_i$ are irreducible.

If $X$ is irreducible, the space $\mathcal{D}$ of linear $\Gamma$-equivariant maps $L:X \to X$
is an associative algebra over $\R$. By Schur's Lemma it is a division algebra
over $\R$, so $\mathcal{D}$ is isomorphic to precisely one of $\R, \C$, and $\Ha$, where
$\Ha$ is the quaternions. We say that $X$ is {\em absolutely irreducible} if
$\mathcal{D} \cong \R$, and {\em non-absolutely irreducible} if
$\mathcal{D} \cong \C$ or $\Ha$.

A component $W$ is {\em $\Gamma$-simple} if and only if either

(a) $W \cong V \oplus V$ where $V$ is absolutely irreducible, or

(b) $W $ is non-absolutely irreducible.

For each irreducible component $Y \subseteq X$, the corresponding
{\em isotypic component} $W_Y$ is defined to be the sum of all irreducible components
isomorphic to $Y$. It  can be written as
\[
W_Y = Y_1 \oplus \cdots \oplus Y_c
\]
where every $Y_i$ is isomorphic to $Y$. Each isotypic component is invariant
under all linear equivariant maps $L:X \to X$.

\subsubsection{Representation-Theoretic Conditions for Equivariant Hopf Bifurcation}

Consider a $\Gamma$-equivariant family of ODEs
\[
\dot x = f(x,\lambda) \qquad (x \in X = \R^k, \lambda \in \Lambda = \R)
\]
Suppose that a Hopf bifurcation from a branch of fully symmetric equilibria 
occurs at $(x_0,\lambda_0)$. Translating
coordinates we may assume $x_0=0, \lambda_0 = 0$, and we do this from now on.
Then $\DD f|_{(0,0)}$ has eigenvalues $\pm \omega \ii$ for $\omega \neq 0$.
This can happen {\em only} if $X$ contains a $\Gamma$-simple component.
The critical eigenspace $E_{\pm \ii\omega}$ is $\Gamma$-invariant,
and generically it is $\Gamma$-simple. Thus `$\Gamma$-simple component' 
is the equivariant analogue
of `simple eigenvalue' in ordinary dynamical systems.  
Moreover, there exists an equivariant coordinate change on $E_{\pm \ii\omega}$
such that, in the new coordinates,
\[
\DD f|_{(0,0)}= J = \Matrix{0 & -\omega I_m \\ \omega I_m & 0}
\]
where $\dim_\R E_{\pm \ii\omega} = 2m$.
The matrix $J$ commutes with the symmetry group
and defines an action of the circle group $\Sone = \R/2\pi \Z$
on $E_{\pm \ii\omega}$ by
\[
\theta x = \exp(\theta J/|\omega|) x
\]
This gives $E_{\pm \ii\omega}$ the structure of a complex vector space,
in which
\[
(r\ee^{\ii \theta})x = r\, \theta x
\]
This is {\em not} the usual complexification $\C \otimes_\R E_{\pm \ii\omega}$;
in particular the dimension of the latter over $\R$ is twice that of $E_{\pm \ii\omega}$.

The group $\Gamma\times \Sone$ now acts on $E_{\pm \ii\omega}$ by
\[
(\gamma,\theta) x = \theta\, \gamma x = \exp(\theta J/|\omega|) \gamma x
\]
This extra circle action induces phase patterns on bifurcating
branches of periodic solutions.

\subsubsection{Symmetry-Breaking}

{\em Spontaneous symmetry-breaking} in a $\Gamma$-equivariant ODE
occurs when the isotropy subgroup of a solution $x(t)$ is smaller than $\Gamma$.

The basic general existence theorem for bifurcating symmetry-breaking
equilibria is the Equivariant Branching Lemma of Cicogna~\cite{Ci81} and Vanderbauwhede~\cite{V80}. See \cite{GSS88, GS02}.
Two kinds of existence theorem for periodic solutions in equivariant systems 
have been developed over the past 30 years. Both are aimed at understanding
the kinds of spatiotemporal symmetries of periodic states that can be 
expected in such systems.
One is the Equivariant Hopf Theorem \cite{GS85a,GS02,GSS88}; the other is the $H/K$ Theorem
\cite{BG01}.  Both have analogues for network dynamics \cite{GS23}. In this paper
only the Equivariant Hopf Theorem is required. It is an analogue of
the Equivariant Branching Lemma and is proved by Liapunov-Schmidt
reduction from an operator equation on loop space \cite{GS85a,GSS88}.

At a non-zero imaginary critical eigenvalue,
the critical eigenspace $E$ supports an action not just of $\Gamma$,
but of $\Gamma \times \Sone$. This action is preserved by Liapunov-Schmidt reduction
\cite[Chapter XVI Section 3]{GSS88}.
The $\Sone$-action on the critical eigenspace is related to, but different from, the
phase shift action; it is determined by the
exponential of the Jacobian $J|_E$ on $E$. Specifically, if
the imaginary eigenvalues are $\pm \ii \omega$ with $\omega \neq 0$  then 
$\theta \in \Sone$ acts on $E$ like the matrix $\exp(\frac{\theta}{|\omega|} J|_E)$.
The $\Sone$-action on solutions is by phase shift, induced via the 
Liapunov-Schmidt reduction procedure. 

The Equivariant Hopf Theorem is analogous to the Equivariant
Branching Lemma, but the symmetry group $\Gamma$ is
replaced by $\Gamma \times \Sone$. 

\begin{definition}\em
\label{D:C-axial}
A subgroup $\Sigma\subset\Gamma\times\Sone$ is {\em $\C$-axial} if $\Sigma$ is 
an isotropy subgroup for the action of $\Gamma\times\Sone$ on $E$ and 
$\dim_\R\Fix(\Sigma)=2$.
\end{definition}

\begin{theorem}[\bf Equivariant Hopf Theorem]
\label{T:EHT}
If the Jacobian has non-real purely imaginary eigenvalues $\pm \ii\omega$, then generically for any
$\C$-axial subgroup $\Sigma \subseteq \Gamma\times\Sone$
acting on the critical eigenspace, there exists a branch 
of periodic solutions with spatiotemporal symmetry group $\Sigma$.
The period tends to $T=\frac{2\pi}{|\omega|}$ at the bifurcation point.
On solutions, $\Sone$ acts by phase shifts.
\end{theorem}
\begin{proof}
See~\cite[Theorem 5.1]{GS85a}, \cite[Theorem 4.9]{GS02}, or
 \cite[Chapter XVI Theorem 4.1]{GSS88}.
\end{proof}

The isotropy subgroup in $\Gamma \times \Sone$ of a periodic solution
is a {\em twisted subgroup}  
\begin{equation}
\label{E:twisted}
H^\phi = \{ (h, \phi(h): h \in H\}
\end{equation}
where $H$ is a subgroup of $\Gamma$ and $\phi:\Gamma \to \Sone$
is a homomorphism. The kernel $K = \ker\phi \subseteq \Gamma$ is
the group of {\em spatial} symmetries, and $H$ is
the group of {\em spatiotemporal} symmetries. The quotient $H/K$
determines the phase shifts of the corresponding phase pattern.
Since $H/K$ embeds in $\Sone$, either $H/K \cong \Sone$
or $H/K \cong \Z_k$ for some $k$. Possible pairs $(H,K)$ are
classified in \cite{BG01}; see also \cite{GMS16}. Not all can be obtained via Hopf bifurcation
from a fully symmetric equilibrium \cite{FG10}.

Equivariant Hopf bifurcations have been studied when the group $\Gamma$ is:

\quad $\Z_n$ \cite[Chapter XVII Section 8]{GSS88}

\quad $\D_n$ \cite{GS86},\cite[Chapter XVIII Sections 1--4]{GSS88}

\quad $\ES_n$ \cite{S96}

\quad $\oo(2)$ \cite[Chapter XVII]{GSS88}

\quad $\ES\oo(2)$ \cite[Chapter XVII Section 8]{GSS88}

\quad $\oo(3)$ \cite[Chapter XVIII Section 5]{GSS88}

\quad  $\ES\oo(n)$ \cite[Chapter XVII Section 5]{GSS88}

\quad $\ES_n\times\oo(2)$ \cite{S96}

\quad Symmetry group of hexagonal lattice \cite[Chapter XVIII Section 6]{GSS88}, \cite{RSW86}

\quad  Symmetry groups of planar lattices \cite{DGSS95,GSS88}

\quad Symmetry group of cubic lattice $\Z^3$ \cite{DS99}

%\subsubsection{$H/K$ Theorem}
%\label{S:H/K}

%Phase patterns need not arise via Hopf bifurcation \cite{GS02}, but 
%a classification is still available.
%A periodic orbit has {\em two} natural symmetry groups. One is
%the group $H$ of transformations that fix the periodic orbit as a set, but change its
%time-parametrisation. The other is the group $K$ of transformations that fix 
%each point in the periodic orbit, and thus leave its time-parametrisation unchanged.
%The pairs of subgroups $(H,K)$ of $\Gamma$ that can arise as
%setwise and pointwise symmetries of a periodic state of a $\Gamma$-equivariant
%ODE are characterised by the $H/K$ Theorem~\cite{BG01,GS02,GS23}. The main
%conditions are that $H$ normalises $K$ with $H/K$ cyclic and
%$K$ is an isotropy subgroup of $\Gamma$. Two further technical conditions apply
%in some cases. 

\subsubsection{Network Analogues}

There are natural analogues of the Equivariant Hopf Theorem and
the $H/K$ Theorem for networks \cite{GS23}. The Network Hopf Theorem
has a rigorous proof; that of the Network $H/K$ Theorem rests
on some `Rigidity Conjectures', as yet proved only in special cases
or under stronger hypotheses. See \cite[Chapters 15, 17]{GS23}. We omit details.

\subsubsection{Time-Reversal}

The next Proposition, about
reversing time in an ODE and a given solution, is trivial but important.

\begin{proposition}
\label{P:time_rev}
Consider a periodic state with period normalised to $2\pi$.
If a phase pattern with phase shifts $\theta_i$ occurs for
an admissible ODE $\dot x = f(x)$, then the reverse pattern
with phase shifts $2\pi-\theta_i$ occurs for
the admissible ODE $\dot x = -f(x)$.
\end{proposition}
\begin{proof}
If $f$ is equivariant or admissible then so is $-f$.
Now assume that $x(t)$ is a periodic solution of $\dot x = f(x)$,
and define $s=-t$ and $y(s) = x(-s)=x(t)$. Then
\[
\frac{\dd}{\dd s} y(s) = -\frac{\dd}{\dd t} y(s) = -\frac{\dd}{\dd t} x(t) =-f(x(t)) = -f(y(s))
\]
Therefore changing $f$ to $-f$ reverses time, and a solution $x(t)$ for
$f$ becomes a solution $y(s)$ for $-f$. A phase
pattern $(\theta_i)$ for $x$ therefore becomes $(-\theta_i)$ for $y$.
Normalising to $\theta_i \in [0,2\pi)$ this becomes $(2\pi-\theta_i)$.
\end{proof}

%This proposition is consistent with the calculation in \cite[Proposition 4.21]{GS02}.

\begin{proposition}
\label{P:reverse_Hopf}
If the admissible ODE $\dot x = f(x,\lambda)$ has a Hopf bifurcation at 
$(x_0,\lambda_0)$, then the time-reversal $\dot x = -f(x,\lambda)$
has a Hopf bifurcation at $(-x_0,\lambda_0)$.
\end{proposition}
\begin{proof}
The Jacobians satisfy
\[
\DD(-f)|_{(-x_0,\lambda_0)} = -\DD f|_{(x_0,\lambda_0)}
\]
The eigenvalues of $\DD(-f)|_{(-x_0,\lambda_0)}$ are minus those
of $\DD f|_{(x_0,\lambda_0)}$. So $\DD(-f)|_{(-x_0,\lambda_0)}$
has eigenvalues $\mp \ii \omega$ if and only if
$\DD f|_{(x_0,\lambda_0)}$
has eigenvalues $\pm\ii \omega$.

Moreover, the non-resonance and eigenvalue crossing conditions 
for $\DD(-f)|_{(-x_0,\lambda_0)}$
are trivially equivalent to those for $\DD f|_{(x_0,\lambda_0)}$.
\end{proof}

\section{Symmetric Unidirectional Rings}
\label{S:SUD}

Consider a symmetric unidirectional $n$-node ring $\UU_n$ with NN
coupling. 

\subsection{1-Dimensional Node Spaces}
Initially we assume that the node spaces are $1$-dimensional,
that is, $\R$. 
Figure \ref{F:rings} (left) shows the case $n=5$. This network has
$\Z_n$ symmetry, generated by an element $\alpha$.
If we number the nodes as $\CC =\{0, 1, \ldots, n-1\}$ then
the symmetry group acts by $\alpha^k(c) = c+k \pmod{n}$.

Admissible ODEs have the form
\begin{equation}
\label{E:Un_ODE}
\dot{x}_c = f(x_c, x_{c+1}) \qquad c \in \CC
\end{equation}
where again we take addition modulo $n$. The linear admissible maps
are all linear combinations $L = a_0 I+a_1 A$ of the  two matrices:
\begin{equation}
\label{E:lin_admiss}
I = \Matrix{1 & 0 & 0 & \cdots & 0 \\
		0 & 1 & 0 & \cdots & 0 \\
		0 & 0 & 1 & \cdots & 0 \\
		\vdots & \vdots &\vdots & \ddots & \vdots \\
		0 & 0 & 0 & \cdots & 1}		
\qquad
A = \Matrix{0 & 1 & 0 & \cdots & 0 \\
		0 & 0 & 1 & \cdots & 0 \\
		0 & 0 & 0 & \ddots & 0 \\
		\vdots & \vdots &\vdots & \ddots & \vdots \\
		1 & 0 & 0 & \cdots & 0}
\end{equation}
where the matrix $A$ is the adjacency matrix for NN coupling.

By symmetry, the eigenspaces of $L$ can be deduced from
the natural permutation representation of $\Z_n$ on $\R^n$.
The isotypic components are the spaces $V_k = \Re(\C\{v_k , v_{n-k}\})$ where
$\Re$ is the real part and 
\[
v_k =  [1, \zeta^k, \zeta^{2k}, \ldots, \zeta^{n-1}]^{\mathrm{T}}
\]
Then $L$ acts on $V_k$ by
\[
v_k \mapsto (a_0+\zeta^k a_1)v_k
\]
so the eigenvalues of $L$ are
\[
\mu_k = a_0+\zeta^k a_1
\]
and each eigenvalue is simple provided that $a_1 \neq 0$.
(If $a_1=0$ no Hopf bifurcation occurs
since all eigenvalues are real; this case is dynamically uninteresting
since the nodes are decoupled.)

For Hopf bifurcation there are two cases: $n$ odd and $n$ even.

\vspace{.1in}
\noindent
{\em Case 1}: $n$ odd.

When $n = 2m+1$ there is one real root of unity $\zeta^0 = 1$,
and the others occur as $m$ complex conjugate pairs $\zeta^r$ and
$\zeta^{n-r} = \bar\zeta^r$ for $1 \leq r \leq m$. The corresponding real eigenspaces 
have dimension 1 and 2, respectively. Hopf bifurcation can occur only
for a 2-dimensional critical eigenspace, so $\zeta^0$ does not occur
and $1 \leq r \leq m$. The other real eigenspaces afford irreducible
representations of $\Z_n$, and these are non-absolutely irreducible, of complex type.

%\subsubsection{First Bifurcation}
A necessary (though not sufficient) condition for a Hopf branch from the trivial
solution to be stable is that it is spawned by the {\em first bifurcation}.
That is, when the corresponding eigenvalues becomes critical (as 
the bifurcation parameter $\lambda$ increases), all
other eigenvalues have negative real part.

When node spaces are 1-dimensional and the ring topology is NN, 
there are strong constraints on the
possible eigenvalues for the first bifurcation to be a Hopf bifurcation.

\begin{theorem}
\label{T:firstHopfNN}
Let $\GG$ be a $\Z_n$-symmetric ring with unidirectional {\rm NN} coupling
and 1-dimensional node spaces. Then the first local bifurcation
is Hopf if and only if $n=2N+1$ is odd. In that case the critical eigenvalues
are the $\rho_k$ for $k= N,N+1$.
\end{theorem}
\begin{proof}
When $a_1 > 0$ the real parts of the eigenvalues are ordered 
in the same manner as the real parts of $\zeta^r$, that is,
\[
\cos \frac{2\pi N}{2N+1} < \cos \frac{2\pi (N-1)}{2N+1} < \cdots < 
	\cos \frac{2\pi}{2N+1} < 1
\]
so the first bifurcation is not Hopf.

When $a_1 < 0$ the order is the reverse of this, so the first bifurcation
occurs for $\zeta^N$ with real part $\cos \frac{2\pi N}{2N+1}$.
It is a Hopf bifurcation.

\vspace{.1in}
\noindent
{\em Case 2}: $n$ even.

When $n = 2N$ there are two real roots of unity: $\zeta^0 = 1$ and $\zeta^N = -1$.
The others occur as $N-1$ complex conjugate pairs $\zeta^r$ and
$\zeta^{n-r} = \bar\zeta^r$ for $1 \leq r \leq N-1$. The corresponding real eigenspaces 
have dimension 1 and 2, respectively. Hopf bifurcation can occur only
for a 2-dimensional critical eigenspace, so $\zeta^0$ and $\zeta^N$ do not occur
and $1 \leq r \leq N-1$. The other real eigenspaces afford irreducible
representations of $\Z_n$, and these are non-absolutely irreducible of complex type.

Again when $a_1 > 0$ the real parts of the eigenvalues are ordered 
in the same manner as the real parts of $\zeta^r$, which are now
\[
-1 < \cos \frac{2\pi (N-1)}{2N} < \cos \frac{2\pi (N-2)}{2N} < \cdots < 
	\cos \frac{2\pi}{2N} < 1
\]
When $a_1 < 0$ the ordering is the reverse. The first bifurcation
occurs for the eigenvalue $a_0 + a_1 \zeta^r$ where $r = 0, N$
respectively. That is, for $a_0 \pm a_1$. Since this is real, the
first bifurcation cannot be Hopf.
\end{proof}

\begin{example}\em
\label{ex:Z51D}
Consider the following $\Z_5$-equivariant ODE in variables $x = (x_0,x_1,x_2,x_3,x_4)$:
\begin{equation}
\label{E:Z51D}
\begin{array}{rcl}
\dot x_0 &=& \lambda x_0 - x_0^3+a x_4 \\
\dot x_1 &=& \lambda x_1 - x_1^3+a x_0 \\
\dot x_2 &=& \lambda x_2 - x_2^3+a x_1 \\
\dot x_3 &=& \lambda x_3 - x_3^3+a x_2 \\
\dot x_4 &=& \lambda x_4 - x_4^3+a x_3 
\end{array}
\end{equation}
where $\lambda$ acts as a bifurcation parameter and $a$ is a coupling strength.
(Again there is an extra symmetry $x \mapsto -x$, which we can remove
with a small quadratic term. See Section \ref{S:ER}.)
There is an equilibrium at $[0,0,0]^\mathrm{T}$ for all $\lambda$. 

Figure \ref{F:Z5hopf} shows a simulation of a rotating wave
for parameter values $\lambda= -1.1, a = -2$.
We omit the eigenvalue analysis, which follows from Theorem \ref{T:firstHopfNN}.

\begin{figure}[h!]
\centerline{%
\includegraphics[width=0.6\textwidth]{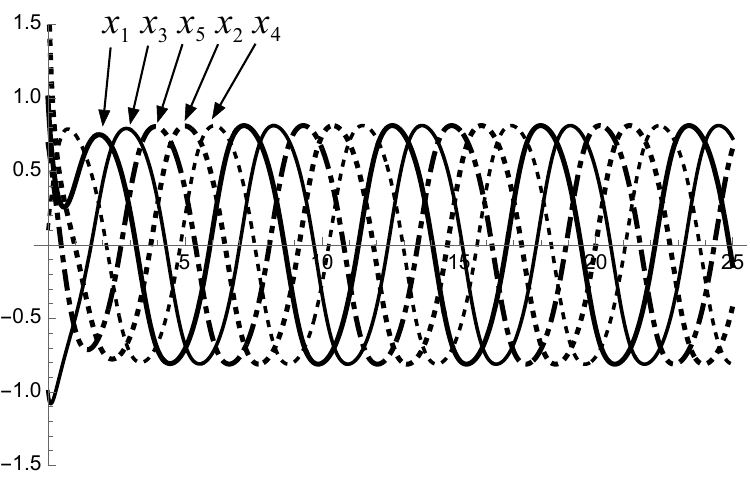}
}
\caption{Hopf bifurcation in a $\Z_5$ ring for the ODE \eqref{E:Z51D}. Horizontal coordinate: $\lambda$. Vertical coordinates: $x_i$ for $0 \leq i \leq 4$. Thick solid = node 1, thick dash = node 2, thin solid = node 3, thin dash = node 4, thick dot-dash = node 5.
Parameters $\lambda= -1.1,a = -2$.}
\label{F:Z5hopf}
\end{figure}

In this simulation the succession of nodes for phase shifts of $T/5$
are $1\ 3\ 5\ 2\ 4$. Equivalently, the phase shifts for nodes in cyclic order $1\ 2\ 3\ 4\ 5$
are $0, 2T/5, 4T/5, T/5, 3T/5$, corresponding to the critical eigenvalue pair
$\mu_2,\mu_3$ (where $\mu_3=\bar\mu_2$). This is in accordance with the `first bifurcation'
condition of Theorem \ref{T:firstHopfNN}. Time reversal leads to the reverse order $0, 3T/5, T/5, 4T/5, 2T/5$.

There is no oscillatory fully synchronous state, since nodes are 1-dimensional.
Phase shifts $0, T/5, 2T/5, 3T/5, 4T/5$ and the reverse can occur, but
not stably.
\end{example}

Again there is a further symmetry $x \mapsto -x$: see Section \ref{S:ER}.

Longer range connections or higher-dimensional node spaces allow other oscillatory 
modes to be the first bifurcation. See Sections \ref{S:LRC} and \ref{S:MNS}.

\section{Hopf Bifurcation in the NN Unidirectional Ring}
\label{S:HBNNUR}

When the symmetry group $\Gamma$ is abelian, and in particular
when it is cyclic, generic Hopf bifurcation in a $\Gamma$-equivariant system
occurs at {\em simple} eigenvalues, because irreducibles are 1-dimensional over
$\C$, hence either 1-dimensional over $\R$ or  2-dimensional over $\R$
with eigenvalues $\pm \ii \omega$ with $\omega \neq 0$. Thus the
classical Hopf Bifurcation Theorem applies. At first sight this might
seem to make the Equivariant Hopf Theorem superfluous, but it is not.
It provides extra information on the spatiotemporal
symmetries of the bifurcating branch of periodic solutions.

To see how this occurs, we
apply the Equivariant Hopf Theorem to the $\Z_n$-symmetric NN-coupled
unidirectional ring. The symmetry group is $\Gamma = \Z_n$ acting on
$\R^n$ by permuting coordinates according to the $n$-cycle $\alpha$,
and $\Z_n = \langle \alpha \rangle$. Extend the base field to $\C$ so that
$\Gamma$ acts on $\C^n$. The $\C$-irreducible subspaces are the
eigenspaces $W_k =\C\{v_k\}$.
The $\R$-irreducible subspaces are the real eigenspaces, which
are the real parts $V_k$ of $\C\{v_k , v_{n-k}\} = w_k+W_{n-k}$. 
The Jacobian $\DD f$ is $a_0I+a_1A$
for $a_0,a_1 \in \R$. The eigenvalues on $V_k$ are $a_0 + a_1\zeta^k$ and $a_0 + a_1\zeta^{n-k}$.
That is,
\[
a_0 + a_1(\cos \frac{2k\pi}{n} \pm \ii \sin\frac{2k\pi}{n})
\]
For Hopf bifurcation, 
\[
0 = a_0 + a_1\cos \frac{2k\pi}{n}
\]
and the eigenvalues are $\pm \ii \omega$ where 
\[
\omega = |a_1 \sin\frac{2k\pi}{n}| = |a_1|
\]
since $0 \leq k \leq n/2$.
Here we take the absolute value because,
when defining the $\Sone$-action, we take the value of
$\omega$ to be positive. (This is important later when we consider
the direction in which the wave rotates.)

Working on $W_k$, the cycle $\alpha$ acts as the matrix $A$,
and on $W_k$ this restricts to 
\[
A_k = \Matrix{\zeta^k & 0  \\ 0 & \zeta^{-k}}
\]
The $\Sone$-action induced from Liapunov-Schmidt reduction 
is given by $\theta$ acting as
\[
M_\theta = \exp(A \theta /\omega) = \exp\Matrix{\pm\ii \theta & 0 \\ 0 & \mp\ii\theta} 
	= \Matrix{\ee^{\pm\ii \theta} & 0 \\ 0 &\ee^{\mp\ii\theta}}
\]
where the sign $\pm$ is that of $b$ and $\mp$ is the opposite sign.
Therefore $(\alpha,\theta)$ acts as
\[
A M_\theta = \Matrix{\ee^{\ii (\frac{2k\pi}{n}\pm\theta)} & 0 \\ 0 &\ee^{ -\ii(\frac{2k\pi}{n}\mp\theta)}}
\]
This gives a 2-dimensional fixed-point space (namely the whole of $V_k$) 
if and only if
\[
\theta = \mp \frac{2k\pi}{n}
\]
When $k \neq 0$ this is the rotating wave condition; when $k=0$ it gives a fully
synchronous standing wave. In either case the solution is fixed by
$(\alpha, \mp \frac{2k\pi}{n})$.

Moreover, we can read off the direction of rotation. It is determined by
the sign of  $\sin\frac{2k\pi}{n}$.

\begin{theorem}
\label{T:ZnHopf}
Let $\GG$ be a unidirectional ring of $1$-dimensional nodes 
with $\Z_n$ symmetry, obeying the ODE \eqref{E:Un_ODE}.
Suppose the critical eigenvalues are those on $V_k$,
where $k \neq 0$ and $k \neq n/2$ for $n$ even. Then generically there
is a bifurcating branch of periodic solutions of the form
\begin{equation}
\label{E:rot_wave_k}
x_j(t) = \phi(t\pm jT/k)
\end{equation}
for a $T$-periodic function $\phi$. (Here we take the same sign $+$ or $-$ throughout.)

The period $T$ converges to $2\pi/|\omega|$ at the bifurcation point. 
The sign in \eqref{E:rot_wave_k} is that of $a_1$.
\end{theorem}

\begin{proof}
The spaces $V_k$ are the irreducible subspaces for the $\Z_n$-action.
When $k \neq 0$ and $k \neq n/2$ for $n$ even, these are non-absolutely
irreducible and of complex type. Over $\R$ they have dimension 2.

The complexified actions of $\alpha \in \Z_n$
and $\theta \in \Sone$ are by the matrices
\[
\alpha = \Matrix{\ee^{2\pi \ii/n} & 0 \\ 0 & \ee^{-2\pi \ii/n}}
\qquad 
\theta = \Matrix{\ee^{\ii \theta} & 0 \\ 0 & \ee^{- \ii\theta}}
\]
Therefore $(\alpha,-2\pi/n)$ acts as the identity. Taking the real part, the group
\[
\tilde\Z_3 = \langle (\alpha,-2\pi/n) \rangle  
\]
is $\C$-axial.
By the equivariant Hopf Theorem, there is a bifurcating branch of periodic states
with spatiotemporal symmetry group $\tilde\Z_n$. On $V_k$ this gives rise
to solutions of the form \eqref{E:rot_wave_k}. (The $\pm$ sign in \eqref{E:rot_wave_k} is explained
in Section \ref{S:DR}.)
\end{proof}

\subsection{Examples \ref{ex:Z31D} and \ref{ex:Z51D} Revisited}
\label{S:ER}

The ODE \eqref{E:Z31D} is also equivariant under $x \mapsto -x$, so the
symmetry group is actually $\Z_3\times\Z_2 \cong \Z_6$. 
In the Equivariant Hopf Theorem, this
extra symmetry identifies with a half-period phase shift, since both
induce minus the identity. We do not get extra solution branches, but
now all solutions are fixed by the `glide reflection' symmetry
$x(t) \mapsto -x(t-T/2)$ of the time series. To destroy this symmetry we can add 
a small quadratic term. The rotating wave state persists.

Similarly, Example \ref{ex:Z51D} is also equivariant under $x \mapsto -x$. So the
symmetry group is actually $\Z_5\times\Z_2 \cong \Z_{10}$. 
Again, solutions are fixed by $x(t) \mapsto -x(t-T/2)$.

\subsection{Longer-Range Coupling}
\label{S:LRC}

We can consider rings with the same cyclic group symmetry
but longer-range coupling, and now the first bifurcation
is less tightly constrained. The results in this section appear to be new.

\begin{figure}[htb]
      \centerline{%
 \includegraphics[width=.3\textwidth]{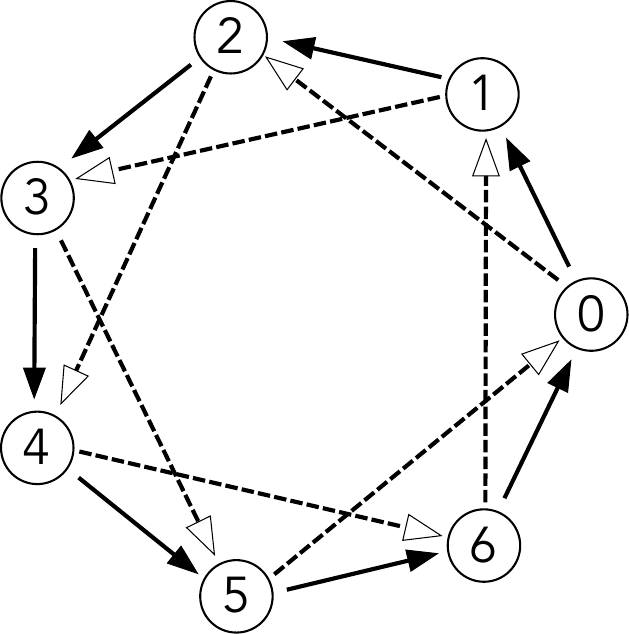}
}
        \caption{A $7$-node ring with NN (solid arrows) and NNN (dashed arrows)
        coupling.}
        \label{F:7gon_non_NN}
\end{figure}

For example, Figure \ref{F:7gon_non_NN} shows a $\Z_7$-symmetric ring
with both nearest neighbour (NN) and next-nearest neighbour (NNN) coupling.
More generally we can consider $r$th nearest neighbour coupling ($r$NN), in which
node $i$ receives an arrow from node $i-r \pmod{n}$. We omit a drawing, 
which would be very cluttered.

The adjacency matrix for identical $r$NN coupling is $A^r$ where $A$ 
is stated in \eqref{E:lin_admiss}. If couplings of all ranges from $0$ to $n-1$
are included, we obtain the {\em group network} for $\Z_n$ acting on
the ring \cite{AS07}. The linear admissible maps are those of the form
\begin{equation}
\label{E:uni_ring_lin}
L = a_0 I + a_1 A + a_2 A^2+ \cdots + a_{n-1} A^{n-1}
\end{equation}
giving the {\em circulant} matrix
\[
L = \Matrix{a_0 & a_1 & a_2 & \ldots & a_{n-1} \\
	a_{n-1} & a_0 & a_1 & \ldots & a_{n-2} \\
	a_{n-2} & a_{n-1} & a_0 & \ldots & a_{n-3} \\
	\vdots & \vdots & \vdots & \ddots & \vdots \\
	a_1 & a_2 & a_3 & \ldots & a_0}
\]
Again the eigenvectors are the $v_k$, but now the eigenvalues are:
\begin{equation}
\label{E:lambda_k}
\mu_k = a_0  + a_1 \zeta^{k} + a_2  \zeta^{2k} + \cdots + a_{n-1}  \zeta^{(n-1)k} 
\end{equation}
The conditions for first bifurcation are now combinatorially more complex,
as as far as we are aware have not previously been studied. We begin
with a simple example.

\begin{example}\em
\label{ex:4nonNN}

Let $n=4$, so $\zeta = \ii$. Theorem \ref{T:firstHopfNN} implies that with only
NN coupling, the first bifurcation cannot be Hopf. We show that this restriction no 
longer holds if longer range couplings occur.

The eigenvalues  $\mu_j$ and their real parts $\rho_j$ are:
\beqn
\mu_0 &=& a_0 + a_1 + a_2 + a_3 \ \qquad \rho_0 = a_0 + a_1 + a_2 + a_3\\
\mu_1 &=& a_0 + a_1\ii - a_2 - a_3\ii \qquad  \rho_1 = a_0  - a_2  \\
\mu_2 &=& a_0 - a_1 + a_2 - a_3 \ \qquad \rho_2 = a_0 - a_1 + a_2 - a_3\\
\mu_3 &=& a_0 - a_1\ii - a_2 + a_3 \ii \qquad  \rho_3 = a_0  - a_2 
\eeqn
There are three distinct real parts $\rho_0,\rho_1,\rho_2$. Transitions in the ordering occur when:
\beqn
\rho_0 = \rho_1: && a_1+2a_2+a_3=0\\
\rho_0 = \rho_2: &&2a_1+a_3=0 \\
\rho_1 = \rho_2: && a_1-2a_2 = 0
\eeqn
The term $a_0$ serves only to translate the values, so it plays no role in the ordering.

These equations define three planes in  $(a_1,a_2,a_3)$ parameter space.
They meet along a common line because the three equations are linearly dependent.
We can therefore take a cross-section, say at $a_1=0$. The result is
Figure \ref{F:4nodeEVorder}. In particular we see that the eigenvalues can occur
in any of the six possible orders, and each order corresponds to a connected
component of the complement of the three transition planes.

\begin{figure}[htb]
      \centerline{%
 \includegraphics[width=.4\textwidth]{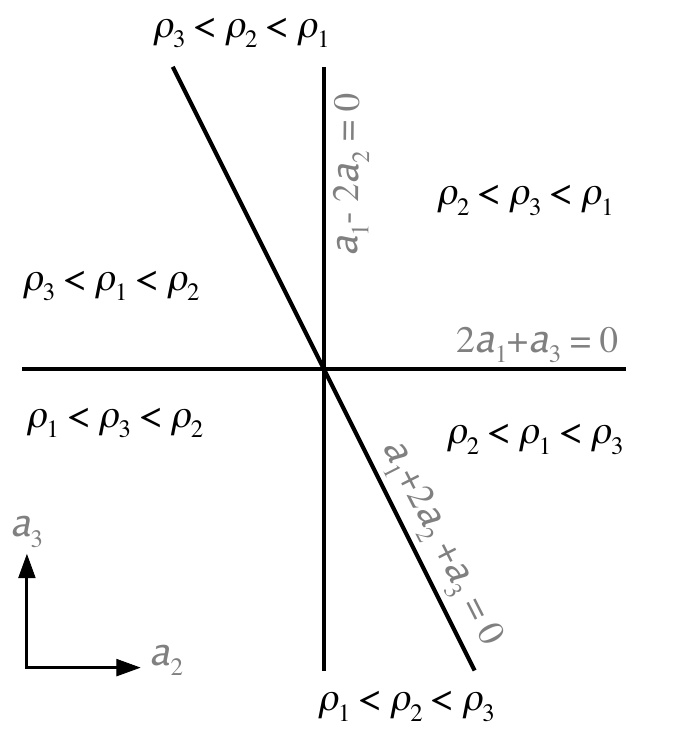}
}
        \caption{Section at $a_1=0$ of  $(a_1,a_2,a_3)$ parameter space,
        showing ordering of real parts of eigenvalues.}
        \label{F:4nodeEVorder}
\end{figure}
\end{example}

\begin{remark}\em
Suppose arrows have range 0, 1, 2, but not 3. Then the above analysis applies 
with $a_3=0$. The transitions occur when
\beqn
\rho_0 = \rho_1: && a_1+2a_2=0\\
\rho_0 = \rho_2: &&a_1=0 \\
\rho_1 = \rho_2: && a_1-2a_2 = 0
\eeqn
and again all possible orderings can occur. Thus connections of all ranges
are not necessary to ensure that all possible orderings can occur.
\end{remark}

%\BLUE{Are ranges $0, 1, \ldots, N$ sufficient for this?}

\subsection{Ordering of Eigenvalues}

When seeking possible orderings of the eigenvalues, the
geometric approach of Example \ref{ex:4nonNN} is too complicated
in general. However, we can prove one useful result: 
If a $\Z_n$-symmetric network has couplings
of all possible ranges, its admissible ODEs admit any ordering of the real parts
of eigenvalues. In particular, any oscillatory mode can be the first bifurcation
for a suitable admissible/equivariant ODE.

To prove this we begin with some standard concepts from ring
theory, and make some important, though obvious and pedantic, distinctions.

Let $t$ be an indeterminate. The {\em polynomial ring} $\C[t]$ consists of
all {\em polynomial expressions} 
\begin{equation}
\label{E:poly_exp}
p = p_0 + p_1 t + p_2 t^2 + \cdots + p_kt^k \qquad (p_j \in \C, 0 \leq j \leq k)
\end{equation}
added and multiplied in the obvious manner. The corresponding {\em polynomial
function} has the form
\[
z \mapsto p_0 + p_1 z + p_2 z^2 + \cdots + p_kz^k \qquad (z \in \C)
\]
We can identify this function with $p$, writing $z \mapsto p(z)$, and use
the same term `{\em polynomial}' for both the formal expression and the
corresponding function.

The ring $\C[t]$ contains the subring 
\[
\R[t]= \{p: p_0, p_1, \ldots,p_k \in \R\}
\]
of polynomials with real coefficients.

The {\em complex conjugate} $\bar p$ of $p$ is defined by
\[
p = \bar p_0 + \bar p_1 t + \bar p_2 t^2 + \cdots + \bar p_kt^k  \in \C[t]
\]
where the bar denotes the usual conjugate in $\C$. An obvious but important
feature is that when $z \in \C$, then in  general $\bar p(z)$ is {\em not}
the same as $\overline{p(z)}$. In fact,
\[
\overline{p(z)} = \bar p(\bar z) \qquad (z \in \C)
\]
and the two sides can differ.

The next result is trivial but crucial:
\begin{lemma}
\label{L:phi+phi_bar}
If $p \in \C[t]$ then $p + \bar p \in \R[t]$.
\end{lemma}
\begin{proof}
With coefficients $p_j$ as in \eqref{E:poly_exp}, 
\[
(p + \bar p) (t) =p(t) + \bar p (t) = \sum_{j=0}^k (p_j + \bar p_j)t^j \in \R[t]
\]
since $p_j + \bar p_j \in \R$.
\end{proof}

With nodes $0, 1, \ldots, n-1$ define
\[
N = \lfloor n/2 \rfloor 
\]
Then the real parts of the eigenvalues are $\mu_0 , \ldots, \mu_N$,
and each singleton or conjugate pair occurs exactly once in this list.
We now prove:

\begin{theorem}
\label{T:anyvalue}
Let $\alpha_0, \ldots, \alpha_N \in \R$. Then there exists a polynomial
$\phi(t)$ of degree $n$ in $\R[t] \subseteq \C[t]$ such that 
\begin{equation}
\label{E:interpolation}
\phi(\zeta^j) = \alpha_j = \phi(\bar \zeta^j) \qquad (0 \leq j \leq N)
\end{equation}
\end{theorem}
\begin{proof}
First observe that if $n$ is even then $\zeta^N = \bar \zeta^N$ so 
\eqref{E:interpolation} imposes $n$ conditions on $\phi$.
By polynomial interpolation there is a polynomial
$\psi \in \C[t]$ such that
\[
\psi(\zeta^j) = \shf\alpha_j = \psi(\bar \zeta^j) \qquad (0 \leq j \leq N)
\]
The coefficients of $\psi$ need not be real, however, so we define
\[
\phi = \psi + \bar\psi
\]
which lies in $\R[t]$ by Lemma~\ref{L:phi+phi_bar}. Now, for $0 \leq j \leq N$,
\beqn
&& \phi(\zeta^j) = \psi(\zeta^j)+ \bar\psi(\zeta^j) = \psi(\zeta^j)+ \overline{\psi(\bar\zeta^j)}
	=\shf\alpha_j +\shf\alpha_j = \alpha_j \\
&& \phi(\bar\zeta^j) = \psi(\bar\zeta^j)+ \psi(\zeta^j) = \psi(\bar\zeta^j)+ \overline{\psi(\zeta^j)}
	=\shf\alpha_j +\shf\alpha_j = \alpha_j 
\eeqn
as required.
\end{proof}

\begin{corollary}
\label{C:anyorder}
The real parts of the eigenvalues of $L$ can occur in any order
for suitable choices of the connection strengths $a_j$.
\end{corollary}
\begin{proof}
Let $\pi$ be any permutation of $\{0,1,\ldots N\}$. Set $\alpha_j = -\pi(j)$
in Theorem \ref{T:anyvalue} and let the $a_j$ in $L$ be the coefficients
of the resulting polynomial $\phi \in \R[t]$. Then
\[
\Re(\mu_j) < \Re(\mu_l) \iff \alpha_j > \alpha_l
\]
\end{proof}

This result also proves that, generically, the eigenvalues are 
all simple and there are no imaginary ones except $\pm \ii\omega$. Thus these
nondegeneracy conditions for Hopf bifurcation hold generically. The
remaining nondegeneracy condition is the eigenvalue crossing condition,
which is generic since it can be arranged by adding the
equivariant and admissible term $\alpha x$ to the ODE,
where $\alpha$ is small.

\subsection{Direction of Rotation}
\label{S:DR}

An issue that has already been raised concerns the direction in which a rotating wave
solution rotates. The $H/K$ classification does not distinguish these;
indeed, for (say) $n=5$ all rotating waves have $H=\Z_5, K=\ONE$, but there 
are four distinct phase patterns with phase shifts $T/5, 2T/5,3T/5,4T/5$.
These occur in pairs $\pm1/5$ and $\pm2/5$. For each irreducible 
representation over $\R$, precisely one from each pair can occur.

The twisted isotropy subgroup $H^{\phi}$ of \eqref{E:twisted} 
does distinguish the rotation directions,
Because $H^{\theta}$ and $H^{-\theta}$ are generally different.
However, in $\Z_n$-equivariant Hopf bifurcation the same critical eigenspace
$U_k = \Re(V_k + V_{n-k})$ supports only one phase pattern,
which may rotate in either direction depending on the admissible ODE.
We give a criterion to decide which of them occurs.

Assume there is a Hopf bifurcation whose critical eigenspace
corresponds to the complex conjugate pair of eigenvectors $v_k$ and $v_{n-k}$,
where $k \neq 0$ and if $n$ is even, $k \neq n/2$. 
The eigenvalues are generically simple, so the standard Hopf Theorem
predicts a {\em unique} branch of periodic states. The Equivariant Hopf Theorem
adds information about the spatio-temporal symmetry of
the solution along this branch. The
branch consists of rotating waves, in which the $n$-cycle $\alpha$
produces a phase shift of $\pm kT/n$, so that
\begin{equation}
\label{E:rot_wave}
x_{j+1}(t) = x_j(t \pm \frac{kT}{n})
\end{equation}
In this case the oscillation at node $j+1$ is phase-shifted by $\mp \frac{2k\pi}{T}$
from the oscillation at node $j$, so the signs are reversed.

The solution branch therefore consists of either a {\em clockwise} 
rotating wave, when the sign in \eqref{E:rot_wave} is negative,
or an {\em anticlockwise}
rotating wave, when the sign in \eqref{E:rot_wave} is positive.
A natural question is: Which sign occurs for a given admissible ODE?

A clockwise rotating wave for $k$ can also be 
interpreted as an anticlockwise one for $n-k$. But it differs from
an anticlockwise one for $k$, so the question is meaningful.

We show that the answer depends on the sign of the imaginary part of
the critical eigenvalue. The proof is based on \cite[Proposition 4.21]{GS02},
which is the special case $n=3$.

Recall from \eqref{E:lambda_k} that the eigenvalues are
\begin{equation}
\label{E:lambda_k2}
\mu_k = a_0  + a_1 \zeta^{k} + a_2  \zeta^{2k} + \cdots + a_{n-1}  \zeta^{(n-1)k} 
\end{equation}
Write $\mu_k = \rho_k + \ii \sigma_k$.
Now
\beqn
\Re(\mu_k) & = & \sum_{j=0}^n a_j \cos \frac{2kj\pi}{n}  = \rho_k\\
\Im(\mu_k) & = & \sum_{j=0}^n a_j \sin \frac{2kj\pi}{n} = \sigma_k
\eeqn
At Hopf bifurcation $\rho_k =0$, so $\mu_k = \ii \sigma_k$.
Moreover, $\sigma_k \neq 0$.

\begin{theorem}
\label{T6.11}
Let $0 \leq k \leq N$.
Consider Hopf bifurcation with critical eigenspace given by the
eigenvalue $\mu_k$. The the direction of rotation is clockwise 
if and only if $\sigma_k < 0$, and anticlockwise 
if and only if $\sigma_k > 0$. 
\end{theorem}
\begin{proof}
Use the eigenvector basis $\{u_k, u_{n-k}\}$ over $\C$ on the complexified
critical eigenspace. Then 
\[
\alpha = \Matrix{\zeta^k & 0 \\ 0 & \zeta^{n-k}} \qquad 
	L|_E = \Matrix{\ii \sigma & 0 \\ 0 & -\ii\sigma}
\] 
The direction of rotation is determined by the kernel of the 
action of the twisted group $\langle (\alpha,\theta)\rangle \subseteq \Z_n \times \Sone$, 
when $\theta$ is the phase shift (either $\frac{2\pi k}{T}$ or $-\frac{2\pi k}{T}$).
That is, we need to know whether
$\alpha = \ee^{2\pi kT/n}$ or $\alpha = \ee^{-2\pi kT/n}$ on $E$.

The $\Sone$-action is given by $\theta$ acts as $\exp(L|_E \theta)$.
This has matrix
\[
\exp(L|_E \theta) = \Matrix{\ee^{\ii \sigma} & 0 \\ 0 & \ee^{-\ii \sigma}}
\]
so either 
\beqn
\sigma &=& \sin \frac{2k\pi}{n} \\
\sigma &=& - \sin \frac{2k\pi}{n} 
\eeqn
Now
\beqn
\sin \frac{2k\pi}{n} > 0 &\mbox{when}& 1 \leq k \leq N \\
\sin \frac{2k\pi}{n} < 0 &\mbox{when}& N+1 \leq k \leq n-1 
\eeqn
We can consider only the case $1 \leq k \leq N$, to avoid the
ambiguity mentioned above. Thus the phase shift is given by the sign of $\sigma$.
If $\sigma>0$ then $\alpha$ identifies with a phase shift of $+kT/n$,
and if $\sigma<0$ then $\alpha$ identifies with a phase shift of $-kT/n$.

\end{proof}

\subsection{Stability}

Since generic $\Z_n$-equivariant Hopf bifurcation occurs at a simple
eigenvalue, the usual stability conditions for Hopf bifurcation without
symmetry apply: see \cite[Chapter 1 Section 4]{HKW81}
or \cite[Chapter VIII Section 4]{GS85}. 
Assume that the branch of equilibria concerned is linearly stable for 
$\lambda < \lambda_0$, so the first bifurcation is the Hopf bifurcation
under discussion. Then
supercritical branches, which exist for $\lambda > \lambda_0$
and $|\lambda-\lambda_0|$ small, are stable near the
bifurcation point, whereas subcritical branches, 
which exist for $\lambda < \lambda_0$
and $|\lambda-\lambda_0|$ small, are unstable.
(Degenerate cases may be neither supercritical nor subcritical. The
results of \cite{GL81} then apply.)

The direction of branching can be computed from the
Liapunov-Schmidt reduction or other dimension reduction methods
such as Poincar\'e-Birkhoff normal form or centre manifold reduction.
See \cite[Chapter 1 Sections 3--6]{HKW81} \cite[Chapter VIII Proposition 3.3]{GS85}.
Since admissible maps contain many quadratic and cubic terms,
it seems likely that the generic behaviour is nondegenerate, with either
a supercritical branch or a subcritical one. 

\section{Multidimensional Node Spaces}
\label{S:MNS}
We sketch how the results for 1-dimensional node spaces generalise
to $l$-dimensional node spaces $\R^l$ for $l > 1$. 
The state space is then $R^{ln}$, which we
can think of as the tensor product $R^l \otimes \R^n$. 
The Jacobian has the form $J=P\otimes I + Q\otimes A$ for arbitrary
$l \times l$ matrices $P, Q$. 

\begin{lemma}
\label{L:PI+QA}
If $y$ is an eigenvector of $P+\zeta^kQ$ with eigenvalue $\rho$,
then $y\otimes v_k$ is an eigenvector of $J$ with eigenvalue $\rho$.

A similar statement holds for generalised eigenvectors.
\end{lemma}
\begin{proof}
Suppose that $(P+\zeta^kQ)y = \rho y$. Then
\beqn
J(y\otimes v_k) &=& (P\otimes I + Q\otimes A)(y\otimes v_k) \\
	&=& Py \otimes v_k + Qy \otimes Av_k \\
	&=& Py \otimes v_k + Qy \otimes \zeta^kv_k \\
	&=& (P+\zeta^kQ)(y\otimes v_k)
\eeqn
The case of a generalised eigenvector is similar.
\end{proof}

\begin{corollary}
\label{C:Jeigen}
The set of eigenvalues of $J$ is the union of the sets of eigenvalues of all
matrices $P+\zeta^kQ$ for $0 \leq k \leq n-1$.
\end{corollary}

Hopf bifurcation then occurs when $P+\zeta^kQ$ has a non-zero purely imaginary
eigenvalue. Since $l \geq 2$ this is always possible for suitable $P,Q$,
so the cases $k=0$ and $k=n/2$ ($n$ even) are no longer excluded.
Provided suitable nondegeneracy conditions hold, the oscillations
for $k=0$ are synchronous; for $k=n/2$, neighbouring nodes are
half a period out of phase; the rest are discrete rotating waves with the
same patterns as in the case $l=1$.

The strong constraints on the first bifurcation no longer apply.
For example, when $n=5$ and $l=1$ the only Hopf bifurcation that can be
the first occurs for $k=2$, and $k=1$ is not possible. When $l=2$, however,
the eigenvalues of $P+\zeta Q$ can be the first critical eigenvalues.
For example, suppose that
\[
Q = \Matrix{1 & 0 \\ 0 & 2} \qquad P = \Matrix{\lambda & 0 \\ 0 & \lambda}
\]
Then the eigenvalues of $\zeta^k Q$, numerically, are
\beqn
k=0 &&  0.5 + 2.39792\,\ii \qquad 0.5 - 2.39792 \,\ii  \\
k=1 && -2.12604 + 1.21652 \,\ii \qquad 2.43506 - 0.265468 \,\ii \\
k=2 && -1.81397 - 1.64606 \,\ii \qquad 1.00495 + 2.23385 \,\ii 
\eeqn
with largest real part $2.43506$ when $k=1$. Thus
the eigenvalues of $\lambda I + \zeta^k Q$ are these with $\lambda$ added,
and occur in the same order of real parts.
As $\lambda$ increases from (say) -3, the first bifurcation
occurs at $\lambda =- 2.43506$, which is a Hopf bifurcation
for the $k=1$ mode of oscillation.

In this example, {\em all} local bifurcations are Hopf.

\section{Symmetric Bidirectional Rings}
\label{S:SBR}

Any $\D_n$-symmetric bidirectional ring is a special case of
the general $\Z_n$-group network unidirectional ring with couplings of
all ranges, discussed in Section \ref{S:LRC}. To obtain $\D_n$
symmetry we require the
input arrows from nodes $i\pm j$ to node $i$ have the same
arrow-type, hence the same connection strength at criticality.

We can therefore read off the linear results---such as the
eigenvalues and eigenvectors---for bidirectional rings
from those for the general $\Z_n$-group network unidirectional ring.
In particular, in \eqref{E:uni_ring_lin} we now have $a_j=a_{n-j}$ for
$0 \leq j \leq n-1$. With 1-dimensional node spaces,
 the eigenvalues are now

\begin{equation}
\label{E:lambda_k_Dn}
\mu_k = a_0  + a_1 (\zeta^{k} + \zeta^{(n-1)k}) + a_2  (\zeta^{2k} +\zeta^{(n-2)k}) + \cdots 
\end{equation}
where now $k$ ranges from $0$ to $\lceil n/2 \rceil$.

All eigenvalues other than those for $k=0$ and $k=n/2$ when $n$ is even are now 
{\em double}. However, the Equivariant Hopf Theorem comes to our rescue
by splitting off simple eigenvalues for $\C$-axial subgroups. The main
new feature is the occurrence of new $\C$-axial subgroups, not generically present
for $\Z_n$ symmetry \cite{GS86,GS02,GS23}. For example, when $n=3$ these $\C$-axial subgroups
are $\Z_2$ subgroups generated by reflection in a symmetry axis, and the
oscillations have on of the following forms:
\beqn
x(t) &=& (u(t),u(t),v(t)) \\
x(t) &=& (u(t),u(t+\frac{T}{2}),v(t))
\eeqn
where $T$ is the period and $v(t)$ has period $\frac{T}{2}$. This called a
{\em multirhythm} state \cite{GS02, GS23, GSS88}.
Conditions for these states to be stable are stated in \cite{GS85,GSS88}
in terms of the Liapunov-Schmidt reduced function.

\subsection{Exotic Patterns}

In \cite[Lemma 4.6, Corollary 4.9]{AS07} it is proved that for 
$\Z_n$- and $\D_n$-symmetric rings with couplings of all ranges,
every balanced colouring is an orbit colouring; that is, determined by 
the fixed-point subspace of a subgroup of the symmetry group. 
The proof of Lemma 4.6 in that paper omits one case, as noted in
\cite[Proposition 26.7]{GS23}, but when this case is taken into account 
the result that every balanced colouring is an orbit colouring
in these networks remains valid \cite[Propositions 26.6, 26.7]{GS23}. 
However, special conditions on
symmetric ring networks can lead to different balanced colourings.
For example,
\cite{GNS04} observe that in a 12-node bidirectional ring with both nearest-neighour
and next-nearest-neighour connections, if these connections have the same arrow-type,
there exists an {\em exotic} synchrony pattern, not given by a subgroup of
the symmetry group $\D_{12}$. 

\begin{figure}[htb]
      \centerline{%
\includegraphics[width=.22\textwidth]{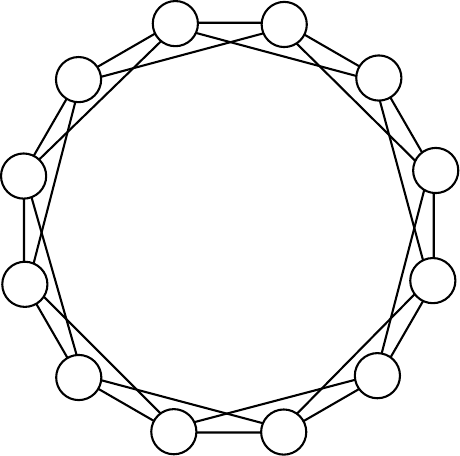} \qquad
\includegraphics[width=.22\textwidth]{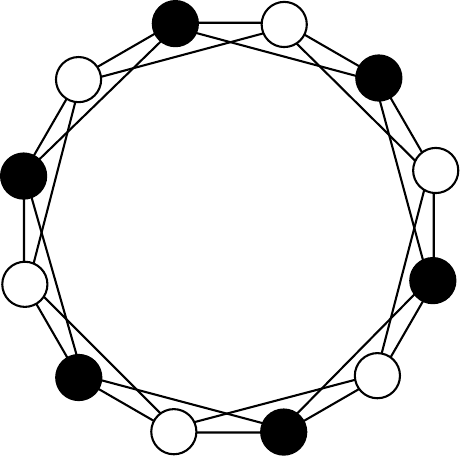} \qquad
\includegraphics[width=.22\textwidth]{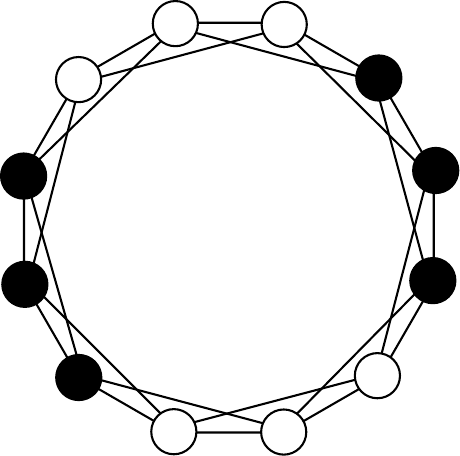} 
        }
        \caption{{\em Left}: 12-node bidirectional ring with NN and NNN 
        coupling, assumed identical. {\em Middle}: Orbit colouring for $v_6$.
        {\em Right}: A balanced colouring that is not an orbit colouring.
        } 
        \label{F:12ring_B}  
\end{figure}

Exotic phase patterns can also occur,
lifted from the corresponding quotient network. Here, the quotient has
$\Z_2$ symmetry. If node spaces have dimension 2 or more, there
can be oscillation patterns in which all nodes of the same colour are synchronous,
but the black nodes are a half period out of phase with the white nodes.
It is shown in \cite[Section 6]{AS06} that this pattern can occur stably as an equilibrium.

\subsection*{Acknowledgements}

I thank Mainak Sengupta for helpful discussions on symmetries of
electronic rectifier circuits. Some of the results in this paper were obtained
in collaboration with Luciano Buono,
Jim Collins, Benoit Dionne,
Marty Golubitsky, Martyn Parker, and David Schaeffer.

\end{document}